\documentclass[12pt]{article}
\usepackage[english]{babel}
\usepackage{amsmath,amsfonts,amssymb,amstext,amsmath,amsthm,amscd}

\usepackage{mathtools}
\usepackage{mathrsfs}
\usepackage{dsfont}

\usepackage{indentfirst}
\usepackage{latexsym}
\usepackage{sidecap}
\usepackage{booktabs}

\usepackage{setspace}
\usepackage{calrsfs}
\usepackage{hyperref}
\hypersetup{
    colorlinks=true,
    linkcolor=red,
    filecolor=magenta,      
    urlcolor=cyan,
}
\usepackage{mathrsfs}
\usepackage{textcomp}
\usepackage{braket}
\usepackage{colortbl}
\usepackage{bbm}

\theoremstyle{plain}
\newtheorem{Theorem}{Theorem}[section]
\theoremstyle{definition}
\newtheorem{Definition}{Definition}[Theorem]
\theoremstyle{remark}
\newtheorem{Remark}[Theorem]{Remark}
\theoremstyle{remark}

\theoremstyle{plain}
\newtheorem{Lemma}[Theorem]{Lemma}
\theoremstyle{plain}

\theoremstyle{plain}

\theoremstyle{remark}

\theoremstyle{remark}

\theoremstyle{remark}
\newtheorem{Hypothesis}[Theorem]{Hypothesis}
\theoremstyle{remark}

\newcommand\RR{\mathbb{R}}
\newcommand\NN{\mathbb{N}}

\DeclareMathAlphabet{\pazocal}{OMS}{zplm}{m}{n}

\newcommand{\OO}{\mathcal{O}}

\usepackage{tikz}
\usetikzlibrary{arrows}

\newcommand{\FT}{\mathscr{F}_t}

\newcommand{\HU}{H^{-1}}

\usepackage{fancyhdr}
\usepackage{calc} 
\begin{document}

\title{Optimal control of the FitzHugh-Nagumo stochastic model with nonlinear diffusion}
\author{Francesco Cordoni$^{a}$ \and Luca Di Persio$^{a}$}
\date{}
\maketitle

\renewcommand{\thefootnote}{\fnsymbol{footnote}}
\footnotetext{{\scriptsize $^{a}$ Department of Computer Science, University of Verona, Strada le Grazie, 15, Verona, 37134, Italy}}
\footnotetext{{\scriptsize E-mail addresses: francescogiuseppe.cordoni@univr.it
(Francesco Cordoni), luca.dipersio@univr.it (Luca Di Persio)}}

\begin{abstract}
We consider the existence and first order conditions of optimality for a stochastic optimal control problem inspired
by the celebrated FitzHugh-Nagumo model, with nonlinear diffusion term, perturbed by a linear multiplicative Brownian-type noise. 
The main novelty of the present paper relies on the application of the {\it rescaling method} which allows us
to reduce the original problem to a random optimal one.
\end{abstract}


\maketitle

\section{Introduction}\label{SEC:Intro}

Consider the following problem

{\footnotesize
\begin{equation} \label{EQN:RDFHN}
\begin{cases}
\partial_t v(t,\xi) = \Delta \gamma(v(t,\xi)) - I^{ion} (v(t,\xi)) - f(\xi) v(t,\xi) + F(t,\xi) +  v(t,\xi) \partial_t W(t) \;, \xi \in \mathcal{O}\\
v(0,\xi) =v_0(\xi) \;, \\
\gamma(v(t,\xi)) = 0 \;, \text{on} \; (0,T) \times \partial \mathcal{O} \;.
\end{cases}
\end{equation}
}

$\gamma: \mathbb{R} \rightarrow \mathbb{R}$ being a monotone, increasing continuous function,  $v=v(t, \xi)$  represents the {\it transmembrane electrical potential}, $\OO \subset \RR^d$, $d=2,3,$ is a bounded and open set with smooth boundary $\partial \OO$. We indicate with $\Delta_\xi$  the Laplacian operator with respect to the spatial variable $\xi$, while $\varepsilon$ and $\delta$ are positive constants representing phenomenological coefficients, $f(\xi)$ is a given external forcing term, while  $I^{ion}$ is the {\it Ionic current} and, according with the FitzHugh-Nagumo model, it equals $I^{ion}(v)= v(v-a)(v-1)$, $v_0$, $w_0 \, \in L^2(\OO)$, namely it represents a cubic non-linearity. Also $F$ is a bounded term needed to treat the general controlled equation in next section.

Equation \eqref{EQN:RDFHN} with linear diffusion, i.e. $\gamma(x)=x$, is the well-known FitzHugh--Nagumo (FHN) equation. FHN equation is a reaction--diffusion equation, first introduced by Hodgkin and Huxley in \cite{HH} and then simplified in \cite{FH,Nag}. The model has been proposed to provide a rigorous, yet simplified, analysis of electrical impulses dyanmics along a nerve axon, see, e.g., \cite{Rin}, where the propagation of the transmembrane potential on the nerve axon is represented by a cubic nonlinear reaction term, possibily perturbed by a noisy one, see, e.g., \cite{BCDP,CDPFHN,Rin,Tuc}. 

The random perturbation represents the effect of noisy input currents within neurons, their source being the random opening/closing actions of  ion channels, see, e.g., \cite{Tuc}. Moreover, in two-dimensional and three-dimensional settings,  equation \eqref{EQN:RDFHN} plays also a relevant role in statistical mechanics,  under the name of \textit{Ginzburg-Landau} equation, as well as concerning phase transition models of \textit{Ginzburg-Landau} type, see, e.g., \cite{Die}. 

The general case where $\gamma$ is a monotone function corresponds to an anomalous--diffusive FitzHugh--Nagumo (FHN) equation, see \cite{Iqb}, also describing phase transitions in porous media, see, e.g., \cite{Pan,Sid}.

\begin{Remark}
In what follows we shall focus on the mathematical setting behind the Stochastic FitzHugh-Nagumo (FHN) model, without
entering into details about the neuro-biological justification of parameters characterizing it. Appropriate details, as well as
in depth analysis of the existing literature on the subject, will be provided later.
\end{Remark}

We assume that $W$ is a $H^{-1}:=H^{-1}(\OO )-$cylindrical Wiener processes, such that
\[
\begin{split}
W(t,\xi) &= \sum_{n= 1}^\infty \mu_n e_n \beta_n(t) \,,\quad t \geq 0\,,\xi \in \OO\,,\\
\end{split}
\]
where $\{\beta_n\}_{n \geq 1}$  is a sequence of mutually independent standard Brownian motion defined on a filtered probability space $\left (\Omega,\mathcal{F},\left \{\mathcal{F}_t\right \}_{t \geq 0},\mathbb{P}\right )$, while  $\{e_n\}_{\geq 1}$ is an orthonormal basis in $\HU$ and $\mu_n \in \RR$.

Since the Laplacian operator $\Delta_\xi$ is a linear operator in $L^2\left (\OO\right )$, and $-\Delta_\xi$ is self--adjoint, then  there exists a complete orthonormal system $\{\bar{e}_k\}_{k\geq 1}$ in $L^2\left (\OO\right )$ of eigenfunctions of $-\Delta_\xi$, and we shall indicate the corresponding sequence of eigenvalues denoted by $\{\bar{\lambda}_k\}_{k\geq 1}$. Therefore, we have
\[
\Delta_\xi \bar{e}_k = -\bar{\lambda}_k \bar{e}_k\,,\quad k \in \NN\;.
\]

Also, we set
\begin{equation}\label{EQN:OpG}
Gv = I^{ion}(v) = v(v-a)(v-1) \;,
\end{equation}
and note that $G$ is monotonically nondecreasing.

The present paper addresses the problem of existence and uniqueness of a strong solution, in a sense to be better specified in a while, to equation \eqref{EQN:RDFHN}. We stress that this is not a trivial problem as the nonlinear operator $\Delta \gamma$ is naturally defined on the space $\HU$ whereas the nonlinear polynomial perturbation $I^{ion}$ is not $m-$accreative on the same space. In order to solve above problem we will transform the original equation, via a \textit{rescaling transformation}, to a random PDE. It turns out that the existence and uniqueness of transformed random PDE can be treated by the theory of nonlinear semigroup in $L^1$.

We will further consider the problem of existence of an optimal control for the nonlinear FHN equation. Again, in order to solve the problem we will apply a \textit{rescaling transformation} to obtain a corresponding random PDE. As already emerged in \cite{BCDP,CDPFHN}, the nonlinear polynomial term implies that standard minimization argument does not apply. Therefore, existence of an optimal control is achieved using \textit{Ekeland’s variational principle}. First order conditions of optimality are given in terms of dual stochastic backward equation, see, e.g, \cite{BCDP,Bre}, whereas, due to the applied \textit{rescaling transformation} are expressed in terms of a random backward dual equation which allows to simplify the setting also giving more insights on the derived optimal controller.

The present paper is structured as follows: Section \ref{SEC:MS} introduces main notation used thorough the paper. Section \ref{SEC:E!} addresses the problem of proving existence and uniqueness for the state equation whereas in Section \ref{SEC:OC} the problem of the existence for an optimal control is considered.

\subsection{Main notations}\label{SEC:MS}

In what follows we will denote by $| \cdot|$, resp. $\langle \cdot,\cdot\rangle$, the norm, resp. scalar product, $\RR^d$. Also, $L^p\left (\OO\right ) =: L^p$, for $1 \leq p \leq \infty$, is the standard space of $p-$Lebesgue measurable function over the domain $\OO \subset \RR^d$, with corresponding norm defined as $|\cdot|_p$. For the case $p=2$, we will further denote by $\langle \cdot,\cdot\rangle_2$ the scalar product in $L^2$. The space $H^1(\OO)=: H^1$ is the Sobolev space $\left \{u \in L^2\,:\,\partial_\xi u(\xi) \in L^2\right \}$, endowed with standard norm $\|u\|_{H^1}^2:= \int_{\OO} \left (|u|^2 + |\nabla u|^2 \right )d\xi$. The dual of the space $H^1$ will be denoted as $\HU$ equipped with corresponding norm $|\cdot|_{-1}$.

Similarly, we will denote by $W^{n,p}(\OO)=:W^{n,p}$, $n \in \NN$, $1\leq p \leq \infty$, the standard Sobolev space of $p-$integrable functions with $p-$integrable $n-$order derivatives. Coherently, $W^{1,p}([0,T];\HU)$ will be the space of absolutely continuous function $u:[0,T]\to \HU$ such that both $u$ and $\frac{d}{dt}u \, \in L^p([0,T];\HU)$. Further, given a Banach space $X$, $L^p([0,T];X)$ is the space of $X-$valued Bochner $p-$integrable functions on the interval $[0,T]$. Also, $C([0,T];X)$, resp. $C^1([0,T];X)$, denotes the space of continuous, resp. continuously differentiable, functions $u:[0,T]\to X$.  

We shall also introduce $C_W([0,T];H)$ the space of all $\HU$--valued $\left (\mathcal{F}_t\right )$--adapted processes such that $X \in C\left ([0,T];L^2\left (\Omega;\HU\right )\right )$, that is $X$ satisfies
\[
\sup_{t \in [0,T]} \mathbb{E}|X(t)|^2_{-1} <\infty\,.
\]

In an analogous manner $L^2_W([0,T];\HU)$ is the space of all $\HU$--valued $\left (\mathcal{F}_t\right )$--adapted processes such that $X \in L^2\left ([0,T];L^2\left (\Omega;\HU\right )\right )$, that is $X$ satisfies
\[
\int_{0}^T \left ( \mathbb{E}|X(t)|^2_{-1}\right )^2 dt<\infty\,.
\]

At last $L^2_W(\Omega;C\left ([0,T];\HU\right ))$ denotes the space of all $\HU$--valued $\left (\mathcal{F}_t\right )$--adapted  and continuous processes such that 
\[
\mathbb{E} \sup_{t \in [0,T]} \left |X(t)\right |^2_{-1}<\infty\,.
\]

Above definition are still in place if instead of $\HU$ we consider a general Hilbert space $H$. It is also known that there is a natural embedding of $L^2_W(\Omega;C\left ([0,T];\HU\right ))$ into the space $C_W([0,T];\HU)$, see, e.g. \cite[Chapter 1]{BDPR}.

We therefore can rewrite equation \eqref{EQN:RDFHN} as
\begin{equation}\label{EQN:MainFHN}
\begin{cases}
dX(t)-[\Delta(\gamma( X(t)))-G(X(t)) + fX(t) + F]dt =XdW(t),\\
X(0) = x_0  \in \HU\, ,\quad  t \in [0,T]\, ,
\end{cases}   
\;.
\end{equation}

We will assume the following to hold.

\begin{Hypothesis}\label{HYP:1}
\begin{description}
\item[(i)] $\gamma:\RR \to \RR$ with $\gamma(0)=0$ is continuous, monotonically non--decreasing and there exists $C>0$ such that
\[
(\gamma(x_1) - \gamma(x_2))(x_1-x_2) \geq C(x_1 - x_2)^2\,,\quad \forall \, x_1,\,x_2 \in \RR\,,
\] 
\item[(ii)] $F \in L^\infty((0,T)\times \OO)$, $\mathbb{P}-$a.s. and it is progressively measurable w.r.t. $(0,T) \times \Omega \times \mathcal{B}(\OO)$; $f \in L^\infty(\OO)$, and $f \geq 0$ $a.e.$ in $\OO$;
\item[(iii)] $W$ is a $\HU:=H^{-1}(\OO)-$cylindrical Wiener processes, that is,
\[
W = \sum_{j=1}^\infty \mu_j e_j \beta_j\,,
\]
with 
\[
\sum_{j=1}^\infty \mu_j^2 |e_j|^2_{L^\infty(\OO)} < \infty\,,
\]
see, \cite[pag. 22]{BDPR}.
\end{description}
\end{Hypothesis}


Then, we can state the notion of solution to equation \eqref{EQN:MainFHN} that we will consider in subsequent analysis.

\begin{Definition}\label{DEF:SS}
Let $x \in \HU$, we say that the process 
\[
X \in L^2_W \left (\Omega;C\left ([0,T];\HU\right )\right ) \cap L^2_W([0,T];L^2)\,,
\]
is a \textit{strong solution} to \eqref{EQN:MainFHN} if $X(t): [0,T] \to \HU$ is $\mathbb{P}-$a.s. continuous and $\forall$ $t \in [0,T]$
\[
X(t) = x + \int_0^t \left (\Delta(\gamma(X(s))) - G(X(s)) + fX(s) + F(s)\right )ds + \int_0^t X(s) dW(s), .
\]
\end{Definition}

\section{Existence for the state equation}\label{SEC:E!}

The main problem in proving existence and uniqueness for a solution to equation \eqref{EQN:MainFHN} is that the operator $G$ in not m-accretive on the space $\HU$ and so basic existence results in \cite{BDPR,BDPR3} are not applicable in the present case. It turns out that the proper space one has to consider to successfully treat equation \eqref{EQN:MainFHN} is the space $L^1$, which, in turn, is not the proper one if one has to deal with SPDEs such as \eqref{EQN:MainFHN}.

To overcome such a stalemate,  we follow \cite{BR,BR2}. In particular, we apply the transformation $X=e^Wy$, which allows to reduce the stochastic equation \eqref{EQN:MainFHN} to a random PDE that can be treated with analytical techniques. In fact, the random equation can be successfully solved  by exploiting the theory of nonlinear semigroup in $L^1$. As noted in \cite{BR}, we have still to face the problem that, because of the non regularity of the term $W$, the general theory cannot be applied straightforward to the resulting random PDE. Therefore, for $\epsilon>0$, we shall consider a suitable sequence of regular approximations $W_\epsilon$ of $W$, to first establish a priori estimates for solutions $y_\epsilon$ of the associated $W_\epsilon-$approximating problem, and then to show that, in  the limit $\epsilon \to 0$, we obtain both existence and uniqueness of the solution for the original equation.

The following theorem constitutes the main result of this section.

\begin{Theorem}\label{THM:E!}
There is a unique strong solution to equation \eqref{EQN:MainFHN} $X=e^Wy$ which satisfies
\[
Xe^{-W} \in W^{1,2}([0,T];\HU)\,,\quad \mathbb{P}-a.s.
\]
\end{Theorem}

In order to prove  Theorem \ref{THM:E!} we need some auxiliary lemmas. In particular, let us then introduce the transformation
\begin{equation}\label{EQN:Resca}
X(t)=e^{W(t)} y(t) \,,\quad t \geq 0\,,
\end{equation}
so that by an application of the It\^{o} formula we obtain the random equation
\begin{equation}\label{EQN:RPDE}
\begin{split}
&\frac{\partial}{\partial t} y + e^{-W} G(e^W y) - e^{-W} \Delta\gamma(e^W y) + fy + \mu y = e^{-W}F \,\\
&y(0,\xi) = x(\xi)\,,\quad \xi \in \OO\,,\\
&y(t) \in H^1_0(\OO)\,,\, t \in (0,T)
\end{split}
\end{equation}
with 
\[
\mu = \frac{1}{2} \sum_{n=1}^\infty \mu_n^2 e_n^2\,,
\]
see, e.g. \cite{BR,BR2,BRrs}.

Following \cite{BR}, we prove the existence of a unique strong solution to equation \eqref{EQN:RPDE} by 
first considering an approximating problem. In particular, let us denote by $\beta^\epsilon (t) := (\beta * \rho_\epsilon)(t)$, where $\rho_\epsilon(t) = \frac{1}{\epsilon}\rho\left (\frac{t}{\epsilon}\right )$ is a standard mollifier and $\rho \in C^\infty_0$, then we have that $\beta^\epsilon \in C^1([0,T];\RR)$. Setting 
\[
W_\epsilon(t,\xi) = \sum_{n= 1}^\infty \mu_n e_n \beta^\epsilon(t) \,,\quad t \geq 0\,,\xi \in \OO\,.\\
\]
we thus have that $W_\epsilon \in C^1([0,T] \times \OO)$. Moreover
\[
W_\epsilon(t,\xi) \to W(t,\xi) \,\quad \mbox{uniformly in} \quad (t,\xi) \in [0,T]\times \xi\,,
\]
as $\epsilon \to 0$.

For each $\epsilon >0$, let us thus consider the approximating equation associated to \eqref{EQN:RPDE}
\begin{equation}\label{EQN:RPDEApp}
\begin{split}
&\frac{\partial}{\partial t} y_\epsilon + e^{-W_\epsilon} G_\epsilon(e^{W_\epsilon} y_\epsilon) - e^{-W_\epsilon} \Delta\left (\gamma(e^{W_\epsilon} y_\epsilon + \epsilon e^{W_\epsilon}y_\epsilon\right ) + fy + \mu y_\epsilon =  e^{-W_\epsilon} F\,\\
&y_\epsilon(0,\xi) = x(\xi)\,,\quad \xi \in \OO\,,\\
&y_\epsilon(t) \in H^1_0(\OO)\,,\, t \in (0,T)
\end{split}
\end{equation}
where $G_\epsilon$ is the Yosida approximation of $G$, that is
\begin{equation}\label{EQN:YosApp}
G_\epsilon := \frac{1}{\epsilon} \left (I - (I + \epsilon G)^{-1}\right )\,, \quad \epsilon >0\,.
\end{equation}

Note that, $G_\epsilon$ is monotoniccaly nondecreasing,  Lipschitzian and 
\[
\lim_{\epsilon \to 0} G_\epsilon(z) = G(z)\,,\, \forall z \in \mathbb{R}
\]
uniformly on compacts.

Defining $z_\epsilon := e^{W_\epsilon}y_\epsilon$,  equation \eqref{EQN:RPDEApp} becomes
\begin{equation}\label{EQN:RPDEAppZ}
\begin{split}
&\frac{\partial}{\partial t} z_\epsilon + G_\epsilon(z_\epsilon) - \Delta\left (\gamma(z_\epsilon)+\epsilon z_\epsilon\right ) + fz + \left ( \mu - \frac{\partial}{\partial t} W_\epsilon\right ) z_\epsilon = F_\epsilon\,, \quad \mbox{in} \, (0,T) \times \OO\,,\\
&z_\epsilon(0,\xi) = x(\xi)\,,\quad \xi \in \OO\,,\\
&\gamma(z_\epsilon(t)) + \epsilon z_\epsilon(t)\in H^1_0(\OO)\,,\, t \in (0,T)
\end{split}
\end{equation}
where 
$F_\epsilon := e^{-W_\epsilon} F$.

\begin{Lemma}
Let $x \in \HU \cap L^1$ with $\gamma(x) \in H^1_0$, then for each $\epsilon >0$ equation \eqref{EQN:RPDEApp} has a unique solution such that
\[
y_\epsilon \in W^{1,\infty} \left ([0,T];\HU\right ) \cap L^\infty \left (0,T;H^1_0\right ).
\]

\end{Lemma}
\begin{proof}
Let us first prove existence and uniqueness of a solution to equation \eqref{EQN:RPDEAppZ} in the space $\HU$. For a fixed $\epsilon >0$, let us define the operator $A:D(A)\subset \HU \to \HU$ as
\begin{equation}\label{EQN:OpA}
\begin{split}
Az &= - \Delta\left (\gamma(z) + \epsilon z\right ) + fz + G_\epsilon(z) + \mu z\,,\\
D(A) &= \left \{z \in L^2 \,:\, \gamma(z) \in H^1_0 \right \}\,.,
\end{split}
\end{equation}

We equip the space $\HU$ with the scalar product
\[
\langle y,z\rangle_{-1} := {}_{H^1}\langle \left (-\Delta\right )^{-1} y,z\rangle_{\HU} \,,
\]
where $\left (-\Delta\right )^{-1} y = x$ indicates  the solution to the Dirichlet problem $-\Delta x = y$ in $\OO$, $x \in H^1_0$.

Taking into account that $G_\epsilon$ is Lipschitz continuous in $L^2$ and since
\[
z \mapsto -\Delta(\gamma(z)+\epsilon z)\,,
\]
is $m-$accreative in the space $\HU$, see, e.g., \cite[p. 68]{Bar1}, we have that, for a suitable $\alpha = \alpha_\epsilon$, it holds 
\[
\langle (A+\alpha I)z - (A+\alpha I)\bar{z},z-\bar{z}\rangle_{-1} \geq0\,
\]
which implies $(A+\alpha I)$ to be accretive in $\HU$. 

Moreover, for $\lambda>0$ sufficiently large, we also have $\mathcal{R}\left ((\lambda + \alpha)I + A\right )=\HU$, so that $A$ is quasi-m-accretive. In other words, for $f \in \HU$ the equation
\begin{equation}\label{EQN:Za}
(\lambda + \alpha)\left (- \Delta\right )^{-1}z + \gamma(z) + \epsilon z + \left (- \Delta\right )^{-1}\left (G_\epsilon(z) + fz -\mu z\right )= \left (- \Delta\right )^{-1}\tilde{f}\,,
\end{equation}
has a unique solution in $z \in L^2$. Indeed, introducing the operators 
\[
\begin{split}
B:L^2 \to L^2\,,\quad Bz := \gamma(z)\,,
\end{split}
\]
and
\[
\begin{split}
\Gamma&: L^2 \to L^2\,,\\
\Gamma z &= (\lambda + \alpha)\left (- \Delta\right )^{-1}z + \left (- \Delta\right )^{-1}\left (G_\epsilon(z) + fz -\mu z\right )\,,
\end{split}
\]
we see that equation \eqref{EQN:Za} can be rewritten as
\begin{equation}\label{EQN:zOper}
\epsilon z + Bz + \Gamma z = \left (- \Delta\right )^{-1}\tilde{f}\,.
\end{equation}
Since $B$ is m-accretive and $\Gamma$ is m-accretive and continuous in $L^2$, it follows, see, e.g., \cite[p.104]{Bar1}, that $\mathcal{R}(\epsilon I + B + \Gamma)=L^2$, so that equation \eqref{EQN:zOper} admits a unique solution $z$ in $L^2$. Moreover, since  $\gamma(z) + \epsilon z \in H^1$ and the inverse map of $z \mapsto \gamma(z) + \epsilon z$ is Lipschitz, then  $z \in D(A)$.
It follows that, applying \cite[Lemma A.1, Corollary A.2]{BR}, see also \cite[Sec. 4]{Bar1},  $z_\epsilon$ is a strong solution to equation \eqref{EQN:RPDEAppZ} in $W^{1;\infty}([0,T];\HU)$. In addition, by \cite[Corollary A.2]{BR}, we also have 
\[
\gamma(z_\epsilon) + \epsilon z_\epsilon - \left (- \Delta\right )^{-1} G_\epsilon(z) \in L^\infty(0,T;H^1_0)\,,
\]
and 
\[
\left |\left (- \Delta\right )^{-1}G_\epsilon(z)\right |_2 \leq C_\epsilon |z|_{-1}\,,
\]
so that, since $z_\epsilon \in W^{1;\infty}([0,T];\HU)$, we obtain
\[
\gamma(z_\epsilon) + \epsilon z_\epsilon \in L^\infty(0,T;L^2)\,,
\]
and consequently $z_\epsilon \in L^\infty(0,T;L^2)$. Moreover we have that
\[
\left ( - \Delta\right )^{-1}G_\epsilon(z_\epsilon) \in L^\infty(0,T;H^1_0)\,,
\]
which implies that $\gamma(z_\epsilon) + \epsilon z_\epsilon \in L^\infty(0,T;H^1_0)$ and consequently $z_\epsilon \in L^\infty(0,T;H^1_0)$.

%
%
\end{proof}

\begin{Lemma}\label{LEM:2}

Taking $x \in D(A)$, then $y_\epsilon \in L^\infty((0,T)\times \OO)$, and it holds
\begin{equation}\label{EQN:IneqY}
\begin{split}
&\sup_{\epsilon} \left \{ |y_\epsilon|_{L^\infty((0,T)\times \OO)} \right \} \leq C(1+|x|_\infty)\,.\\
\end{split}
\end{equation}
\end{Lemma}
\begin{proof}
Let $\alpha \in C^1([0,T])$, such that $\alpha(0) = 0$ and $\alpha'\geq 0$. Then, defining  $M:= (1+|x|_\infty)$, we have
\[
\begin{split}
&\frac{\partial}{\partial t} \left (y_\epsilon- M \alpha(t)\right ) + e^{-W_\epsilon} \left (G_\epsilon(e^{W_\epsilon} (y_\epsilon))- G_\epsilon(e^{W_\epsilon}(M +  \alpha(t)))\right ) - f y_\epsilon +\\
&- e^{-W_\epsilon} \Delta\left (\gamma(e^{W_\epsilon} y_\epsilon)+\epsilon e^{W_\epsilon}y_\epsilon\right ) + \\
&+ e^{-W_\epsilon}\Delta \left ( \gamma( e^{W_\epsilon}(M+\alpha(t)))+ \epsilon e^{-W_\epsilon}(M+\alpha(t))\right )+\\
&+ \mu (y_\epsilon-M-\alpha(t)) = \tilde{F}_\epsilon -\alpha'\,,
\end{split}
\]
with 
\[
\begin{split}
\tilde{F}_\epsilon &:= - e^{-W_\epsilon}G_\epsilon(e^{W_\epsilon}(M +  \alpha(t)))-(M+\alpha(t))+ f y_\epsilon +\\
&+ e^{-W_\epsilon} \Delta \gamma(e^{W_\epsilon}(M+\alpha(t))) + \epsilon(M+\alpha(t))e^{-W_\epsilon}\Delta(e^{W_\epsilon})+ e^{-W_\epsilon}F\,,
\end{split}
\]
with $\alpha$ such that $F_\epsilon - \alpha' \leq 0$.

Following \cite[Lemma 3.3]{BR}, we  first assume that
\begin{equation}\label{EQN:CondL1}
\frac{\partial}{\partial t} y_\epsilon \,,\quad \Delta\left (\gamma(e^{W_\epsilon} y_\epsilon)+\epsilon e^{W_\epsilon}y_\epsilon\right ) \in L^1((0,T)\times \OO)\,,
\end{equation}
then, denoting by
\[
\begin{split}
J(t)&:= - \int_{\OO} e^{-W_\epsilon}\left [ \left (\Delta(\gamma(e^{W_\epsilon}y_\epsilon)+\epsilon e^{W_\epsilon}y_\epsilon\right ) \right .+\\
&- \left .\Delta\left (\gamma(e^{W_\epsilon}(M+\alpha(t))) + \epsilon e^{W_\epsilon}(M+\alpha(t))\right )\right ]sign(y_\epsilon -M -\alpha(t))^+ d\xi\,,
\end{split}
\]
we have
\[
\int_0^t J(s) ds \geq - (\gamma'(e^{|W|_\infty}M)+1)e^{|W|_\infty}(|\Delta W|_\infty + |\nabla W|_\infty^2)\int_0^t |(y_\epsilon - (M+\alpha(s))))^+|_1 ds\,.
\]
Moreover, by hypothesis \ref{HYP:1}, it follows that $G_\epsilon$ is monotone, so that

{\footnotesize
\[
\begin{split}
&\int_0^t \int_{\OO} e^{-W_\epsilon} \left (G_\epsilon(e^{W_\epsilon} (y_\epsilon))- G_\epsilon(e^{W_\epsilon}(M +  \alpha(t)))\right ) sign(y_\epsilon - M - \alpha(s))^+ ds d\xi = \\
&\int_0^t \int_{\OO} e^{-W_\epsilon} \left (G_\epsilon(e^{W_\epsilon} (y_\epsilon))- G_\epsilon(e^{W_\epsilon}(M +  \alpha(t)))\right ) sign(G_\epsilon(y_\epsilon) - G_\epsilon(M + \alpha(s)))^+ ds d\xi \geq 0\,,
\end{split}
\]
}

and since
\[
\begin{split}
&\int_{\OO} \frac{\partial}{\partial t}\left (y_\epsilon - M -\alpha(s)\right ) sign(y_\epsilon - M -\alpha(s))^+ d\xi =\\
&\frac{d}{dt}|(y_\epsilon(t) - M -\alpha(t))^+|_1 \,, \mbox{a.e.} t \in (0,T)\,,
\end{split}
\]
by \cite[Lemma 3.3]{BR}, we conclude that
\[
|(y_\epsilon(t) - M -\alpha(t))^+|_1 =0\,,
\]  
if $\tilde{F}_\epsilon \leq \alpha'$ a.e. in $(0,T)\times \OO$. Moreover, for a suitable $\alpha$, it also holds 
\[
y_\epsilon \leq M+\alpha(t)\,,\quad  \mbox{a.e.}\, \mbox{in} (0,T)\times \OO\,,
\]
and
\[
y_\epsilon \geq -M-\alpha(t)\,,\quad  \mbox{a.e.}\, \mbox{in} (0,T)\times \OO\,,
\]
and  inequalities \eqref{EQN:IneqY} follows.

Using the approximating scheme described in \cite[Lemma 3.3]{BR}, we have \eqref{EQN:IneqY} without requiring the condition \eqref{EQN:CondL1}, and the claim follows.
\end{proof}

\begin{Lemma}\label{LEM:1}
Let $x \in D(A)$, then there exists an increasing function $C:[0,\infty) \to (0,\infty)$ such that
\[
\sup_{t \in [0,T]} |y_\epsilon(t)|^2_2 + \int_0^T \int_{\OO} |\nabla \gamma(e^{W_\epsilon}y_\epsilon)|^2 d\xi ds \leq C(1+|x|_\infty)\,,\quad \forall \epsilon \in (0,1]\,.
\]
\end{Lemma}
\begin{proof}
In what follows we will use the following
\[
\begin{split}
&\int_{\OO} j(y_\epsilon(t)) d\xi = \int_0^t {}_{\HU} \left \langle \frac{d y_\epsilon}{ds}(s),\gamma(y_\epsilon(s))\right \rangle_{H^1_0} ds + \int_{\OO} j(x) d\xi\,,\\
&{}_{\HU} \left \langle \Delta \gamma(e^{W_\epsilon}y_\epsilon), e^{-W_\epsilon}\gamma(y_\epsilon) \right \rangle= - \int_{\OO} \nabla \gamma(e^{W_\epsilon}y_\epsilon) \cdot \nabla(e^{-W_\epsilon}\gamma(y_\epsilon)) d\xi\,,
\end{split}\,,
\]
with $j(r) = \int_0^r \gamma(s)ds$, $r \in \RR^+$.

Thus, multiplying equation \eqref{EQN:RPDEApp} by $\gamma(y_\epsilon)$ and integrating over $(0,t) \times \OO$ we obtain
\begin{equation}\label{EQN:BouGg}
\begin{split}
&\int_{\OO} j(y_\epsilon(t)) d\xi + \int_0^t \int_{\OO} \left [ \nabla \gamma(e^{W_\epsilon}y_\epsilon) + \epsilon \nabla(e^{W_\epsilon}y_\epsilon)) \cdot \nabla(\gamma(y_\epsilon)e^{-W_\epsilon})\right ] d\xi ds \leq \\
&\leq \int_{\OO} j(x) d\xi - \int_0^t \int_{\OO} \left (e^{-W_\epsilon} G_\epsilon(e^{W_\epsilon} y_\epsilon) + f y_\epsilon\right )\gamma(y_\epsilon) d\xi ds\,.
\end{split}
\end{equation}

Concerning the last integral in the right hand side of equation \eqref{EQN:BouGg}, we have recalling that $G_\epsilon$ is the Yosida approximant of $G$ and using the monotonicity of $\gamma$ and $G_\epsilon$ that Lemma \ref{LEM:2},

\begin{equation}\label{EQN:BouGg2}
\begin{split}
&\int_0^t \int_{\OO} \left (e^{-W_\epsilon} G_\epsilon(e^{W_\epsilon} y_\epsilon) + f y_\epsilon \right )\gamma(y_\epsilon) d\xi ds \geq  \\
&\geq C\int_0^t \int_{\OO} e^{-W_\epsilon} G_\epsilon(e^{W_\epsilon} y_\epsilon)y_\epsilon d\xi ds \geq  \\
&\geq C\int_0^t \int_{\OO} e^{-W_\epsilon} y_\epsilon d\xi ds \,,
\end{split}
\end{equation}
while the other terms in equation \eqref{EQN:BouGg} can be studied as done in \cite[Lemma 3.3]{BR}, so that the claim follows by Lemma \ref{LEM:2}. 
\end{proof}

\begin{Lemma}\label{LEM:E!RPDE}
There is a unique solution to equation \eqref{EQN:RPDE} with

\begin{equation}\label{EQN:BoundSolY}
\begin{split}
&y \in W^{1,2}([0,T];\HU) \cap L^\infty((0,T) \times \OO)\,,\\
&y \in L^2(0,T;\HU)\,,\quad \gamma(e^W y) \in L^2(0,T;H^1_0)\,.
\end{split}
\end{equation}

Moreover, the process $y$ is $\left (\mathcal{F}_t\right )_{t \geq 0}-$adapted.

\end{Lemma}
\begin{proof}
Let us first prove uniqueness. Let $y_1$ and $y_2$ be two solutions to equation \eqref{EQN:RPDE}, and let $\bar{y} := y_1 - y_2$. Then it holds 
\begin{equation}\label{EQN:RPDE!}
\begin{split}
&\frac{\partial}{\partial t} \bar{y} + e^{-W} \left (G(e^W y_1)-G(e^W y_2)\right ) + f \bar{y}+ \\
&\qquad - e^{-W} \Delta\left (\gamma(e^W y_1) - \gamma(e^W y_2)\right ) + \mu \bar{y} = 0\, \quad \mbox{in} \quad (0,T) \times \OO\\
&\bar{y}(0,\xi) = 0\,,\quad \xi \in \OO\,.\\
\end{split}
\end{equation}

We can rewrite equation \eqref{EQN:RPDE!} as
\begin{equation}\label{EQN:RPDE!2}
\begin{split}
&\frac{\partial}{\partial t} \bar{y} + (-\Delta)(\bar{y}\eta)= - e^{-W} \left (G(e^W y_1)-G(e^W y_2)\right ) - e^{-W} \Delta(e^{-W})\bar{y}\eta+\\
&- f \bar{y} - 2 \nabla(e^{-W}) \cdot \nabla(e^W \bar{y}\eta)- \mu \bar{y} = 0\,,\\
\end{split}
\end{equation}
where we have denoted for short
\[
\eta := 
\begin{cases}
\frac{\left (\gamma(e^W y_1) - \gamma(e^W y_2)\right )}{e^W \bar{y}} & \{(t,\xi) \,:\, \bar{y}(t,\xi) \not = 0\}\,,\\
0 & \{(t,\xi) \,:\, \bar{y}(t,\xi) = 0\}\,.\\
\end{cases}
\]

Multiplying equation \eqref{EQN:RPDE!2} by $(-\Delta)^{-1}\bar{y}$, we obtain

\begin{equation}\label{EQN:StimaF}
\begin{split}
&\frac{1}{2}|\bar{y}|_{-1}^2 + \int_0^t \int_{\OO} \eta \bar{y}^2 ds d\xi =\\
&= \int_0^t \int_{\OO} e^{-W} \left (G(e^W y_1)-G(e^W y_2)\right )(-\Delta)^{-1}\bar{y} ds d\xi +\\
&- \int_0^t \int_{\OO} e^W \Delta(e^{-W})\bar{y}\eta (-\Delta)^{-1}\bar{y} ds d\xi+\\
&-2 \int_0^t \int_{\OO} \nabla(e^{-W}) \cdot \nabla(e^W \bar{y}\eta)(-\Delta)^{-1}\bar{y} ds d\xi+\\
&-\int_0^t \int_{\OO} f  \bar{y}(-\Delta)^{-1}\bar{y} ds d\xi - \int_0^t \int_{\OO} \mu  \bar{y}(-\Delta)^{-1}\bar{y} ds d\xi +\,.
\end{split}
\end{equation}

Concerning the first integral in the right hand side of equation \eqref{EQN:StimaF}, we have using the fact that, for $\alpha \in [0,1]$ it holds
\[
G(e^W y_1)-G(e^W y_2) = G'(\alpha e^W y_1 + (1-\alpha) e^W y_2) e^W \bar{y}\,,
\]
we infer that, denoting for short $\vartheta = \alpha e^W y_1 + (1-\alpha) e^W y_2$,
\[
\begin{split}
&\left | \int_{\OO} e^{-W} \left (G(e^W y_1)-G(e^W y_2)\right )(-\Delta)^{-1}\bar{y} d\xi\right |=\\
&=\left | \int_{\OO} G'(\vartheta)\bar{y}(-\Delta)^{-1}\bar{y} d\xi\right | \leq C |\bar{y}|_2 |\bar{y}|_{-1}\,,
\end{split}
\]
whereas other terms can be treated as in \cite[Theorem 2.2]{BR}. So that, we have
\[
\frac{d}{dt}|\bar{y}|_{-1}^2 \leq C |\bar{y}|_{-1}^2 \,, \, \mbox{a.e.}\, t >0\,,
\]
from which it follows that $\bar{y} = 0$, and, by Lemma \ref{LEM:1}, it holds
\[
|y(t)|_\infty + \int_0^t \int_{\OO} |\nabla \gamma(y(s))|^2 d\xi ds \leq C(1+ |x|_\infty)\,,
\]
so that, see \cite[Theorem 2.2]{BR}, we further have
\[
\begin{split}
&y \in W^{1,2}([0,T];\HU) \cap L^\infty((0,T) \times \OO)\,.\\
\end{split}
\]

As regard existence, by Lemma \ref{LEM:1}--\ref{LEM:2}, we have that $(\gamma(e^{W_\epsilon}y_\epsilon))$ is bounded in $L^2(0,T;H^1_0)$, $(y_\epsilon)$ is bounded in $L^\infty(0,T;L^2) \cap L^\infty ((0,T) \times \OO) \cap L^2(0,T;H^1_0)$, and  $\left (\frac{d y_\epsilon}{dt}\right )$ is bounded in $L^2(0,T;\HU)$. Thus, by Aubin compactness theorem, $(y_\epsilon)$ is compact in each $L^2(0,T;L^2(\OO))$.
It follows that, for fixed $\omega \in \Omega$, along a subsequence, which we still denote by $\{\epsilon\} \to 0$ for the sake of clarity, we have
\begin{equation}\label{EQN:ConSubs}
\begin{split}
y_\epsilon \to y & \quad \mbox{strongly}\, \mbox{in} \, L^2((0,T);L^2)\,,\\
& \quad\mbox{weak-star}\, \mbox{in} \, L^2((0,T);L^2)\,,\\
& \quad\mbox{strongly}\, \mbox{in} \, L^\infty((0,T)\times \OO)\,,\\
& \quad\mbox{weakly}\, \mbox{in} \, L^2((0,T);H^1_0)\,,\\
\gamma(e^{W_\epsilon}y_\epsilon) \to \eta & \quad \mbox{weakly}\, \mbox{in} \, L^2((0,T);H^1_0)\,,\\
\frac{d y_\epsilon}{dt} \to \frac{d y}{dt} & \quad \mbox{weakly}\, \mbox{in} \, L^2((0,T);\HU)\,,\\
W_\epsilon \to W & \quad \mbox{in} \, C((0,T)\times \OO)\,.\\
\end{split}
\end{equation} 

Since the map $z \mapsto \gamma(z)$ is maximal monotone, by \eqref{EQN:ConSubs} we have that $\eta = \gamma(e^Wy)$.

Then, since it holds 
\[
|(1+\epsilon G)^{-1}(e^{W_\epsilon} y_\epsilon) - e^{W_\epsilon} y_\epsilon|\leq \epsilon |G_\epsilon(e^{W_\epsilon} y_\epsilon)| \leq C \epsilon\,,
\]
and 
\[
(1+\epsilon G)^{-1}(e^{W_\epsilon} y_\epsilon) \to y  \quad \mbox{strongly}\, \mbox{in} \, L^2((0,T);L^2)\,,\\
\]
then, for $\epsilon \rightarrow 0$, we get
\begin{equation}\label{EQN:ConvGe}
\begin{split}
G_\epsilon \left (e^{W_\epsilon} y_\epsilon\right ) \to \zeta & \quad \mbox{weakly}\, \mbox{in} \, L^2((0,T);H^1_0)\,.\\
\end{split}
\end{equation}

Thus, again from the fact that $G:\RR \to \RR$ is maximal monotone it follows that it is also closed and therefore we have that $\zeta = G(e^W y)$.

Therefore, by letting $\epsilon \to 0$, from equation \eqref{EQN:RPDEApp} we obtain
\[
\begin{split}
&\frac{dy}{dt} + e^{-W} G(e^Wy) + fy - e^{-W} \Delta \gamma(e^Wy) + \mu y = e^{-W} F\,,\quad \mbox{in} \, (0,T) \times \OO\,,\\
&y(0)=x\,.
\end{split}
\]

Then, by the  uniqueness result already proved,
we also have that the sequence $(y_\epsilon)$ is independent of $\omega \in \Omega$, implying that 
$y$ is $\left (\mathcal{F}_t\right )-$adapted, ending the proof.
\end{proof}

We can finally prove that it exists a unique strong solution $X$ to equation \eqref{EQN:MainFHN} which satisfies
\[
Xe^{-W} \in W^{1,2}([0,T];\HU)\,\quad \mathbb{P}-a.s.
\]

\begin{proof}[Proof of Theorem \ref{THM:E!}]
Using \cite[Lemma 8.1]{BR2} we have the equivalence between the stochastic PDE \ref{EQN:MainFHN} and the random PDE \ref{EQN:RPDE} via the rescaling transformation \ref{EQN:Resca}, so that existence and uniqueness of a solution $X$ in the sense of Definition \ref{DEF:SS} follows by Lemma \ref{LEM:E!RPDE}.
\end{proof}

\section{The optimal control problem}\label{SEC:OC}

In this section we will focus the attention to a controlled version of equation \eqref{EQN:RDFHN}. We  denote by $\mathcal{X}=L^2_{ad}\left ((0,T) \times  \OO\right )$ the space of all $\FT - $adapted processes $u : [0,T ] \to \RR^d$, and 
%
we consider the following optimal control problem 
\begin{equation}\label{EQN:P}
  \tag{P}
\mbox{Minimize} \,  \mathbb{E} \left [ \int_0^T \int_{\OO} \left |v(t,\xi) -v_1(\xi)\right |^2  + \frac{\alpha}{2} |u(t,\xi)|^2 d\xi dt + \int_{\OO} \left |v(T,\xi) -v_2(\xi)\right |^2 d\xi \right ]\,,
\end{equation}
subject to $u \in \mathcal{U}$ and

{\footnotesize
\begin{equation} \label{EQN:CFHN}
\begin{cases}
\partial_t v(t,\xi) - \Delta \gamma(v(t,\xi)) + I^{ion} (v(t,\xi)) + f(\xi) v(t,\xi) = u(t,\xi) +  v(t,\xi) \partial_t W(t) \;, \mbox{in}\, (0,T)\times \mathcal{O}\\
v(0,\xi) =v_0(\xi)\,,\quad \xi \in \OO \;, \\
v(t,\xi) = 0 \;, \text{on} \; (0,T) \times \partial \mathcal{O} \;.
\end{cases}
\end{equation}
}

Here
\[
\begin{split}
\mathcal{U} &:= \left \{ u \in L^2_{ad}((0,T)\times \OO \times \Omega) \right .\,:\\ 
&\left . \qquad \quad   \,|u(t,\xi,\omega)|\leq M\,\quad \mbox{a.e.} (t,\xi,\omega) \in (0,T)\times \OO\times \Omega\right \}\,,
\end{split}
\]
$M>0$ being a suitable constant, while $v_1$, $v_2 \in L^2_{\mathcal{F}_0}(\Omega)$ and $\alpha>0$ are given.


In what follows we are going to treat the problem (P) by a {\it rescaling procedure} which allows us to reduce it to a random optimal control problem.

\begin{Theorem}\label{THM:E!4}
Let hypothesis \ref{HYP:1} holds, then, for $T$ sufficiently small, there exists at least one optimal pair $ \left (u^*,v^*\right ) $ solution to problem \eqref{EQN:P}.
\end{Theorem}
\begin{proof}
As in Section \ref{SEC:E!},
we will apply the rescaling transformation $y := e^{-W}v$ so that the optimal control problem \eqref{EQN:P} reads
\begin{equation}\label{EQN:P2}  \tag{P2}
\begin{split}
\mbox{Minimize} \, &\mathbb{E} \left [\int_0^T \int_{\OO} \left |e^W y(t,\xi) -v_1(\xi)\right |^2  + \frac{\alpha}{2} |u(t,\xi)|^2 d\xi dt \right ] +\\
&+ \mathbb{E} \left [\int_{\OO} \left |e^W y((T,\xi) -v_2(\xi)\right |^2 d\xi \right ] \,,
\end{split}
\end{equation}
subject to
\begin{equation} \label{EQN:CFHNPDE}
\begin{split}
&\partial_t y - e^{-W} \Delta \gamma(e^W y) + e^{-W} G(e^W y)+fy+\mu y = e^{-W}u \;, \mbox{in}\, (0,T)\times \mathcal{O}\,,\\
& y = 0 \,,\quad \mbox{in} \quad (0,T) \times \partial \OO\,.
\end{split}
\end{equation}

Existence and uniqueness for a solution to equation \eqref{EQN:CFHNPDE} follows from Lemma \ref{LEM:E!RPDE}.

Applying Ekeland’s variational principle, see, e.g. \cite{Eke} or also \cite{CDPFHN,BCDP}, there exists a sequence $\{u_\epsilon\}\subset \mathcal{U}$ such that
\begin{equation}\label{novemberrain1}
\begin{split}
\Psi(u_\epsilon) &\leq \inf \{ \Psi(u) \, ; u \in \mathcal{U}\}+ \epsilon\, ,\\
\Psi(u_\epsilon) &\leq \Psi(u) + \sqrt{\epsilon}\left |u_\epsilon - u\right |\, ,\quad \forall \, u \in \mathcal{U}\, ,
\end{split}
\end{equation}
or equivalently it holds
\begin{equation}\label{novemberrain2}
u_\epsilon = \arg \min_{u \in \mathcal{U}} \{  \Psi(u) + \sqrt{\epsilon}\left |u_\epsilon - u\right |_{\mathcal{U}} \}\, .
\end{equation}

By the standard maximum principle for problem \eqref{novemberrain1}, we have
\begin{equation}\label{EQN:MaxPEk}
\begin{cases}
&\partial_t y_\epsilon - e^{-W} \Delta \gamma (e^W y_\epsilon) + e^{-W} G(e^W y_\epsilon) + fy_\epsilon+\mu y_\epsilon = \frac{1}{\alpha} e^{-W}(p_\epsilon + \theta_\epsilon)\,\\
&\partial_t p_\epsilon + e^{W} \gamma'(e^W y_\epsilon) \Delta(e^{-W}p_\epsilon) + e^W G'(e^W y_\epsilon)p_\epsilon - f p_\epsilon - \mu p_\epsilon = 2(e^W y_\epsilon -v_1)\,\\
&y_\epsilon (0) = y_0 \,,\quad p_\epsilon(T) = 2 (e^W y_\epsilon(T) - v_2)\,,\\
\end{cases}
\end{equation}

where $|\theta_\epsilon|_{L^2(\Omega \times \OO \times (0,T))}\leq \sqrt{\epsilon}$.
Indeed, by \eqref{novemberrain1}, it follows $\Psi'(u_\epsilon) = \theta_\epsilon$, yelding \eqref{EQN:MaxPEk}.


By equation \eqref{EQN:MaxPEk} we have $\mathbb{P}-a.s.$, 
\begin{equation}\label{EQN:MaxPEkDiff}
\begin{split}
&\partial_t \left (y_\epsilon - y_\lambda\right ) - e^{-W} \Delta \left ( \gamma (e^W y_\epsilon) - \gamma (e^W y_\lambda)\right ) + \\
&- e^{-W} \left ( G(e^W y_\epsilon)- G(e^W y_\lambda)\right )+f (y_\epsilon - y_\lambda)+\mu (y_\epsilon - y_\lambda) =\\
&= \frac{1}{\alpha} e^{-W}\left ((p_\epsilon -p_\lambda) + (\theta_\epsilon-\theta_\lambda)\right )\,,\\
\end{split}
\end{equation}
hence, multiplying equation \eqref{EQN:MaxPEkDiff} by $(-\Delta)^{-1}(y_\epsilon - y_\lambda)$ and integrating over $\OO$, it holds
\begin{equation}\label{EQN:Est1}
\begin{split}
&\frac{1}{2}|y_\epsilon - y_\lambda|_{-1}^2 + \int_0^t \int_{\OO} \eta (y_\epsilon - y_\lambda)^2 d\xi ds =\\
&= - \int_0^t \int_{\OO} e^{-W} \left (G(e^W y_\epsilon)-G(e^W y_\lambda)\right )(-\Delta)^{-1}(y_\epsilon - y_\lambda)d\xi ds +\\
&- \int_0^t \int_{\OO} e^W \Delta(e^{-W})(y_\epsilon - y_\lambda)\eta (-\Delta)^{-1}(y_\epsilon - y_\lambda) d\xi ds+\\
&-2 \int_0^t \int_{\OO} \nabla(e^{-W}) \cdot \nabla(e^W (y_\epsilon - y_\lambda)\eta)(-\Delta)^{-1}(y_\epsilon - y_\lambda) d\xi ds+\\
&-\int_0^t \int_{\OO} f (y_\epsilon - y_\lambda)(-\Delta)^{-1}(y_\epsilon - y_\lambda) d\xi ds +\\
&-\int_0^t \int_{\OO} \mu  (y_\epsilon - y_\lambda)(-\Delta)^{-1}(y_\epsilon - y_\lambda) d\xi ds +\\
&+\int_0^t \int_{\OO} \frac{e^{-W}}{\alpha} (p_\epsilon -p_\lambda)(-\Delta)^{-1}(y_\epsilon - y_\lambda) d\xi ds +\\
&+\int_0^t \int_{\OO} \frac{e^{-W}}{\alpha} (\theta_\epsilon -\theta_\lambda)(-\Delta)^{-1}(y_\epsilon - y_\lambda) d\xi ds \,,
\end{split}
\end{equation}
where $\eta$ is defined as in \eqref{EQN:StimaF}.

While, the first four integrals in the right hand side of equation \eqref{EQN:Est1} can be bounded similarly as done proving Lemma \ref{LEM:E!RPDE}, see \eqref{EQN:StimaF}, the last two terms can be treated exploiting the Young inequality 
\begin{equation}\label{EQN:EstPT}
\begin{split}
\left |\int_{\OO} \frac{e^{-W}}{\alpha} (p_\epsilon -p_\lambda)(-\Delta)^{-1}(y_\epsilon - y_\lambda) d\xi \right | &\leq C |p_\epsilon -p_\lambda|_2^2 |y_\epsilon - y_\lambda|_{-1} \leq \\
&\leq C_1 |p_\epsilon -p_\lambda|_2^2 + C_2 |y_\epsilon - y_\lambda|_{-1}\,,\\
\left |\int_{\OO} \frac{e^{-W}}{\alpha} (\theta_\epsilon -\theta_\lambda)(-\Delta)^{-1}(y_\epsilon - y_\lambda) d\xi \right | &\leq C |\theta_\epsilon -\theta_\lambda|_2^2 |y_\epsilon - y_\lambda|_{-1} \leq \\
&\leq C |y_\epsilon - y_\lambda|_{-1} + \epsilon + \lambda\,,
\end{split}
\end{equation} 
as to  obtain
\begin{equation}\label{EQN:Est2}
\begin{split}
&\frac{1}{2}\frac{d}{dt} |y_\epsilon(t) - y_\lambda(t)|_{-1}^2 \leq C_1 |y_\epsilon(t) -y_\lambda(t)|_{-1}^2 + C_2 |p_\epsilon(t) -p_\lambda(t)|_{2}^2 + \epsilon + \lambda \,.
\end{split}
\end{equation}

Applying the Gronwall lemma and taking the mean value, we have
\begin{equation}\label{EQN:Est2bis}
\begin{split}
&\mathbb{E} |y_\epsilon(t) - y_\lambda(t)|_{-1}^2 \leq C \mathbb{E} \int_0^t |p_\epsilon(s) -p_\lambda(s)|_{2}^2 ds + \epsilon + \lambda \,,.
\end{split}
\end{equation}
Regarding the second equation in \eqref{EQN:MaxPEk},  we obtain 
\begin{equation}\label{EQN:MaxPEkDiff2}
\begin{split}
&\partial_t (p_\epsilon - p_\lambda) + e^{W} \left (\gamma'(e^W y_\epsilon) \Delta(e^{-W} p_\epsilon) - \gamma'(e^W y_\epsilon) \Delta(e^{-W}p_\lambda)\right )+ \\
&+e^W \left (G'(e^W y_\epsilon)p_\epsilon - DG(e^W y_\lambda) p_\lambda \right )-f (p_\epsilon - p_\lambda) -\mu (p_\epsilon - p_\lambda) = 2(e^W (y_\epsilon -y_\lambda))\,.\\
\end{split}
\end{equation}
Then, multiplying  equation \eqref{EQN:MaxPEkDiff2} by $(p_\epsilon - p_\lambda)$ and integrating over $\OO$, we obtain
\[
\begin{split}
&\frac{1}{2}|p_\epsilon(t) - p_\lambda(t)|_2^2 = |y_\epsilon(T) - y_\lambda(T)|_2^2 +\\
&\int_t^T \int_{\OO} e^{W} \left (\gamma'(e^W y_\epsilon) \Delta(e^{-W}p_\epsilon) - \gamma'(e^W y_\lambda) \Delta(e^{-W}p_\lambda)\right )(p_\epsilon(s) - p_\lambda(s)) d\xi ds + \\
&+\int_t^T\int_{\OO} e^W \left (G'(e^W y_\epsilon)p_\epsilon(s) - G'(e^W y_\lambda)p_\lambda(s)\right )(p_\epsilon(s) - p_\lambda(s)) d\xi ds+\\
 &-\int_t^T \int_{\OO}\mu (p_\epsilon(s) - p_\lambda(s))^2 d\xi ds+\\
 &-\int_t^T \int_{\OO} 2(e^W (y_\epsilon -y_\lambda))(p_\epsilon(s) - p_\lambda(s)) d\xi ds \,.
\end{split}
\]
Rearranging terms above, we further have 
\begin{equation}\label{EQN:Est3}
\begin{split}
&\frac{1}{2}|p_\epsilon(t) - p_\lambda(t)|_2^2 = |y_\epsilon(T) - y_\lambda(T)|_2^2 +\\
&\int_t^T \int_{\OO} e^{W} \gamma'(e^W y_\epsilon) \left (\Delta(e^{-W}p_\epsilon) - \Delta(e^{-W}p_\lambda)\right )(p_\epsilon(s) - p_\lambda(s)) d\xi ds + \\
  &+\int_t^T \int_{\OO} e^{W} \Delta(e^{-W}p_\lambda) \left (\gamma'(e^W y_\epsilon) - \gamma'(e^W y_\lambda))\right )(p_\epsilon(s) - p_\lambda(s)) d\xi ds + \\
&+\int_t^T\int_{\OO} e^W G'(e^W y_\epsilon) \left (p_\epsilon(s) - p_\lambda(s)\right )^2 d\xi ds+\\
&+\int_t^T\int_{\OO} e^W p_\lambda(s) \left (G'(e^W y_\epsilon) - G'(e^W y_\lambda)\right )(p_\epsilon(s) - p_\lambda(s)) d\xi ds+\\
 &-\int_t^T \int_{\OO}\mu (p_\epsilon(s) - p_\lambda(s))^2 d\xi ds+\\
 &-\int_t^T \int_{\OO} 2(e^W (y_\epsilon -y_\lambda))(p_\epsilon(s) - p_\lambda(s)) d\xi ds \,. \\ 
\end{split}
\end{equation}
Therefore, by using the Young inequality, we have
\begin{equation}\label{EQN:Est3bis}
\begin{split}
&|p_\epsilon(t) - p_\lambda(t)|_2^2 \leq |y_\epsilon(T) - y_\lambda(T)|_2^2 +\\
&C \int_t^T |p_\epsilon(s) - p_\lambda(s)|_2^2 ds + C \int_t^T |p_\lambda (s)|_2 |p_\epsilon(s) - p_\lambda(s)|_2 ds +\\
&+C \int_t^T |y_\epsilon(s) - y_\lambda(s)|_2 |p_\lambda (s)|_2 |p_\epsilon(s) - p_\lambda(s)|_2 ds \leq \\
&\leq C\left ( \int_t^T |p_\epsilon(s) - p_\lambda(s)|_2^2 ds + \int_t^T |y_\epsilon(s) - y_\lambda(s)|_2^2 ds\right )\,,
\end{split}
\end{equation}
where we have used $C$ to denote possibly different constants as to simplify notation.

Taking the expectation in equation \eqref{EQN:Est3bis} and combining it with equation \eqref{EQN:Est2bis}, we thus have
\begin{equation}\label{EQN:EstComb}
\begin{split}
&\mathbb{E} |y_\epsilon(t) - y_\lambda(t)|_{-1}^2 + |p_\epsilon(t) - p_\lambda(t)|_2^2 \leq \\
&\leq C \left ( \mathbb{E} \int_0^t |p_\epsilon(s) -p_\lambda(s)|_{2}^2 ds + \epsilon + \lambda \right ) +\\
&+C\left ( \int_t^T |p_\epsilon(s) - p_\lambda(s)|_2^2 ds + \int_t^T |y_\epsilon(s) - y_\lambda(s)|_2^2 ds\right )\,,
\end{split}
\end{equation}
so that, if $T$ is small enough, we can infer that
\begin{equation}\label{EQN:EstFinal}
\begin{split}
&\mathbb{E} |y_\epsilon(t) - y_\lambda(t)|_{-1}^2 + |p_\epsilon(t) - p_\lambda(t)|_2^2 \leq C \left (\epsilon + \lambda \right ) \,,
\end{split}
\end{equation}
implying that  $(y_\epsilon,p_\epsilon)$ is a Cauchy sequence, therefore
,along a subsequence still denoted by $\{\epsilon\} \to 0$ for the sake of clarity, we have
\begin{equation}\label{EQN:ConSubs2}
\begin{split}
y_\epsilon \to y & \quad\mbox{weakly}\, \mbox{in} \, L^2((0,T);H^1_0)\,,\\
p_\epsilon \to p & \quad\mbox{weakly}\, \mbox{in} \, L^2((0,T)\times \OO)\,,\\
u_\epsilon := \frac{1}{\alpha} e^{-W}(p_\epsilon + \theta_\epsilon) \to u^* & \quad\mbox{weakly}\, \mbox{in} \, L^2((0,T)\times \OO \times \Omega)\,.
\end{split}
\end{equation} 

Letting then $\epsilon \to 0$ in the first equation in \eqref{EQN:MaxPEk}, we have
\[
\partial_t y^* - e^{-W} \Delta \gamma (e^W y^*) - e^{-W} G(e^W y^*)+\mu y^* = e^{-W} u^*\,,\\
\]
hence, since $\Psi$ is lower--semicontinuous, previous computations give us:
\[
\Psi(u^*) = \inf_{u \in \mathcal{U}} \Psi(u)\,,
\]
and the claimed existence result follows.

\end{proof}

\begin{Theorem}[Necessary condition of optimality]\label{THM:NC}
Let be $(v^*,u^*)$ an optimal pair for problem \eqref{EQN:P}, then if $\alpha > 0$ it holds
\[
u^*(t,\xi) = \frac{1}{\alpha} P_U(-e^{-W}p(t,\xi))\,,\quad \mbox{a.e. on} \quad (0,T) \times \OO \times \Omega\,,
\]
where $p$ is the solution to the dual backward equation \eqref{EQN:DualBack} and
\begin{equation}\label{EQN:Proj}
P_U(v) = 
\begin{cases}
M & v \geq M\,,\\
v & |v| \leq M\,\\
-M & v \leq M\,.\\
\end{cases}
\end{equation}
\end{Theorem}
\begin{Remark}
 We would like to underline that in literature about stochastic control problem, the first order conditions of optimality (the {\it Pontryagin maximum principle})
are expressed in terms of dual stochastic backward equation, see, e.g, \cite{BCDP,Bre}. Here, instead, optimality conditions are given in terms of a random backward dual equation which allows to simplify the setting also giving more insights on the derived optimal controller.
\end{Remark}
\begin{proof}
We provide the result exploiting the rescaling transformation $y := e^{-W}X$, hence proving necessary condition for the problem \eqref{EQN:P2}.

Let $(y^*,u^*)$ be an optimal pair for problem \eqref{EQN:P2}, therefore we have that for any $u \in \mathcal{U}$, defining $u^\lambda := u^* + \lambda \bar{u} = u^* + \lambda (u - u^*)$, $\lambda \geq 0$, by the optimality of $u^*$ it must hold,
\[
\frac{1}{\lambda} \left (\Psi(u^\lambda) - \Psi(u^*)\right ) \geq 0\,.
\]

By the G\^{a}teaux differentiability of $\Psi$ it follows, taking the limit as $\lambda \to 0$,
\begin{equation}\label{EQN:VarCont}
\begin{split}
&\mathbb{E}\left [ \int_0^T \int_{\OO} \left(e^W y^*(t,\xi) -v_1(\xi)\right) z(t) + \alpha u^* \bar{u} d\xi dt \right ]+\\
&+ \mathbb{E}\left [\int_{\OO} \left (e^W y^*(T,\xi) -v_2(\xi)\right) z(T) d\xi \right ] \geq 0\,,
\end{split}
\end{equation}
being $z$ the solution to the system in variation defined as
\begin{equation}\label{EQN:SysVar}
\begin{cases}
&\dot{z}(t) - \Delta \left (\gamma'(e^W y^{*}) z(t)\right ) + G'(e^W y^{*})z(t) + \mu z(t) = e^{-W} u^*  \,,\\
&\gamma'(e^W y^{*}) z(t) \in H^1_0(\OO)\,,\, t \in (0,T)\\
& z(0)=0\,.
\end{cases}
\end{equation}
Therefore, introducing the backward dual system
\begin{equation}\label{EQN:DualBack}
\begin{split}
\dot{p}(t) &= - \Delta \gamma' (e^W y^{*})p - G'(e^W y^{*})p_\lambda (t) + \mu p(t) + 2(e^W y^{*} - v_1)\,,\\
p(T) &= 2(y^*(T)-v_2)\,,
\end{split}
\end{equation}
and exploiting equations \eqref{EQN:SysVar}--\eqref{EQN:DualBack} together with equation \eqref{EQN:VarCont},
we have 
\begin{equation}\label{EQN:FeedB}
\begin{split}
\mathbb{E}\left [ \int_0^T \int_{\OO} \left \langle e^{-W}p(t) + \alpha u^*, \bar{u} \right \rangle dt \right ]\geq 0\,,
\end{split}
\end{equation}
which gives
\[
u^*(t,\xi) = \frac{1}{\alpha} P_U(-e^{-W}p(t,\xi))\,,\quad \mbox{a.e. on} \quad (0,T) \times \OO \times \Omega\,,
\]
where $P_U$ 
is the projection operator 
defined  in \eqref{EQN:Proj}.
\end{proof}

\begin{Theorem}[The bang--bang principle]
Let be $(v^*,u^*)$ an optimal pair for problem \eqref{EQN:P} and let $\alpha = 0$, then it holds
\begin{equation}\label{EQN:BangBang}
u^*= 
\begin{cases}
-M & \text{ if } p > 0\\
\in [-m,M] & \text{ if } p = 0\\
M & \text{ if } p <0 \,.
\end{cases}
\end{equation}
where $p$ is the solution to the dual backward equation \eqref{EQN:DualBack}.
\end{Theorem}
\begin{proof}
Proceeding as in Theorem \ref{THM:NC} with obtain the equivalent of equation \eqref{EQN:FeedB}to be
\begin{equation}\label{EQN:FeedBBB}
\begin{split}
\mathbb{E}\left [ \int_0^T \int_{\OO} \left \langle e^{-W}p (t) , \bar{u} \right \rangle dt\right ] \geq 0\,,
\end{split}
\end{equation}
which yields equation \eqref{EQN:BangBang}, and the claim follows.
\end{proof}
\begin{Remark}
 By \eqref{EQN:BangBang} it follows that if $\mid v^* -v_1 \mid >0 $ a.e. on $(0,T) \times \OO \times \Omega$ then the optimal controller $u^*$ is a {\it bang-bang} controller, namely $\mid u^* \mid =M$ a.e. on $(0,T) \times \OO \times \Omega$. 
\end{Remark}


\newpage
\begin{thebibliography}{}
\addcontentsline{toc}{chapter}{Bibliografy}
\bibliographystyle{plain}
\thispagestyle{empty}


\bibitem{Alb1} S. Albeverio and L. Di Persio, Some stochastic dynamical models in neurobiology: recent developments Europena communications in mathematical and theoretical biology, vol. 14 , 2011 , pp. 44-53

\bibitem{Alb3} S. Albeverio, L. Di Persio and E. Mastrogiacomo. Small noise asymptotic expansions for stochastic PDE's, I. The case of a dissipative polynomially bounded non linearity. Tohoku Mathematical Journal Vol.63, No.4 (2011): 877-898.

\bibitem{Bar1} V. Barbu. Nonlinear differential equations of monotone types in Banach spaces. Springer Monographs in Mathematics. Springer, New York,  2010.

\bibitem{Bar2} V. Barbu. Analysis and control of nonlinear infinite-dimensional systems. Mathematics in Science and Engineering, 190. Academic Press, Inc., Boston, MA,  1993.

\bibitem{BCDP} V. Barbu, F. Cordoni and L. Di Persio. Optimal control of stochastic FitzHugh-Nagumo equation.  International Journal of Control (2015): 1-11.

\bibitem{BDPR} V. Barbu, Viorel, G. Da Prato, and M. Röckner. Stochastic porous media equations. Vol. 2163. Springer, 2016.

\bibitem{BDPR3} V. Barbu, Viorel, G. Da Prato, and M. Röckner. "Existence of strong solutions for stochastic porous media equation under general monotonicity conditions." The Annals of Probability 37.2 (2009): 428-452

\bibitem{BR} Barbu, Viorel, and Michael Röckner. "Nonlinear Fokker–Planck equations driven by Gaussian linear multiplicative noise." Journal of Differential Equations 265.10 (2018): 4993-5030.

\bibitem{BR2} Barbu, Viorel, and Michael Röckner. "An operatorial approach to stochastic partial differential equations driven by linear multiplicative noise." Journal of the European Mathematical Society, 17-7 (2015): 1789-1815.

\bibitem{Bar3} V. Barbu and M. Iannelli. Optimal control of population dynamics. Journal of optimization theory and applications, Vol.102, No.1,(1999) 1-14.

\bibitem{Bar4} V. Barbu and T. Precupanu. Convexity and optimization in Banach spaces. Springer Science $\&$ Business Media, 2012.

\bibitem{BRrs} V. Barbu and M. Röckner. On a random scaled porous media equation. Journal of Differential Equations 251.9 (2011): 2494-2514.

\bibitem{Ben} P. Benilan, H. Brezis, M.G. Crandall, \emph{A semilinear equation in $L^1 (\mathbb {R}^N)$}, Annali della Scuola Normale Superiore di Pisa - Classe di Scienze 2.4 (1975): 523-555.

\bibitem{BMZ} S. Bonaccorsi, C. Marinelli and G. Ziglio. Stochastic FitzHugh-Nagumo equations on networks with impulsive noise. Electron. J. Probab, Vol.13 (2008): 1362-1379.

\bibitem{Bon} S. Bonaccorsi S. and E. Mastrogiacomo, Analysis of the Stochastic FitzHugh-Nagumo system, {\it Inf. Dim. Anal. Quantum Probab. Relat. Top.}, Vol.11No.3 pp.427-446, 2008.

\bibitem{Bre} Breiten, Tobias, and Karl Kunisch. "Riccati-Based Feedback Control of the Monodomain Equations With the Fitzhugh--Nagumo Model." SIAM Journal on Control and Optimization 52.6 (2014): 4057-4081.

\bibitem{CRT} E. Casas, C. Ryll and F. Tröltzsch. Sparse optimal control of the Schlögl and FitzHugh–Nagumo systems. Computational Methods in Applied Mathematics Vol.13, No.4 (2013): 415-442.

\bibitem{CDP1} F. Cordoni and L. Di Persio. Small noise asymptotic expansion for the infinite dimensional Van der Pol oscillator. \textit{International Journal of Mathematical Models and Method in Applied Sciences} \textbf{9} (2015)

\bibitem{CordoniDPJMAA2017} F. Cordoni and L.  Di Persio.
Stochastic reaction-diffusion equations on networks with dynamic time-delayed boundary conditions
\textit{ Journal of Mathematical Analysis and Applications}, 451 (1), pp. 583-603, 2017.

\bibitem{CDP2017} F. Cordoni and L. Di Persio. Gaussian estimates on networks with dynamic stochastic boundary conditions. \textit{Infinite Dimensional Analysis, Quantum Probability and Related Topics} 20.01 (2017): 1750001.

\bibitem{CDPFHN} F. Cordoni and L. Di Persio; Optimal control for the stochastic FitzHugh-Nagumo model with recovery variable, Evolution Equations and Control Theory (2018).

\bibitem{DapK} G. Da Prato. Kolmogorov equations for stochastic PDEs. Advanced Courses in Mathematics. CRM Barcelona. Birkhäuser Verlag, Basel, 2004.

\bibitem{Dap} G. Da Prato and J. Zabczyk. Stochastic equations in infinite dimensions. Vol. 152. Cambridge university press, 2014.

\bibitem{Dap2} G. Da Prato and J. Zabczyk. Second order partial differential equations in Hilbert spaces. London Mathematical Society Lecture Note Series, 293. Cambridge University Press, Cambridge, 2012

\bibitem{DapE} G. Da Prato and J. Zabczyk. Ergodicity for infinite-dimensional systems. London Mathematical Society Lecture Note Series, 229. Cambridge University Press, Cambridge, 1996.

\bibitem{Die} H.W. Diehl. The theory of boundary critical phenomena. International Journal of Modern Physics B, Vol.11, No.30 (1997): 3503-3523.

\bibitem{Eke} I. Ekeland. On the variational principle. Journal of Mathematical Analysis and Applications, Vol.47, No.2 (1974): 324-353.

\bibitem{Fuh} M. Fuhrman and C. Orrieri. Stochastic maximum principle for optimal control of a class of nonlinear SPDEs with dissipative drift. SIAM Journal on Control and Optimization 54.1 (2016): 341-371.

\bibitem{FT} M. Fuhrman and G. Tessitore. Nonlinear Kolmogorov equations in infinite dimensional spaces: the backward stochastic differential equations approach and applications to optimal control. Annals of probability (2002): 1397-1465.

\bibitem{FH} R. FitzHugh. Impulses and physiological states in theoretical models of nerve membrane. Biophysical Journal, Vol.1, (1961): 445–466

\bibitem{HH} A. L. Hodgkin, and A. F. Huxley. A quantitative description of membrane current and its application to conduction and excitation in nerve. The Journal of physiology 117.4 (1952): 500-544.

\bibitem{Iqb} Iqbal, Naveed, Ranchao Wu, and Biao Liu. "Pattern formation by super-diffusion in FitzHugh–Nagumo model." Applied Mathematics and Computation 313 (2017): 245-258.

\bibitem{MaDPZ} C. Marinelli, L. Di Persio and G. Ziglio. Approximation and convergence of solutions to semilinear stochastic evolution equations with jumps. Journal of Functional Analysis, Vol. 264, Issue 12, 15 (2013): 2784-2816

\bibitem{Nag} J. Nagumo, S. Arimoto, and S. Yoshizawa. An active pulse transmission line simulating nerve axon. Proceedings of the Institute of Radio Engineers, Vol.50, No.10 (1962): 2061–2070.

\bibitem{Pan} Pankratov, L. "Homogenization of the ginzburg—landau heat flow equation in a porous medium." Applicable analysis 69.1-2 (1998): 31-45.

\bibitem{Pit} S. Pitchaiah and  A. Armaou. Output feedback control of the FitzHugh-Nagumo equation using adaptive model reduction, Proceedings of the 49th IEEE Conference on Decision and Control,  Atlanta, GA, 2010. 

\bibitem{Rin} Ringkvist, Mattias, and Yishao Zhou. "On the dynamical behaviour of FitzHugh–Nagumo systems: revisited." Nonlinear Analysis: Theory, Methods $\&$ Applications 71.7-8 (2009): 2667-2687.

\bibitem{Sca} Scarpa, Luca. "Optimal Distributed Control of a Stochastic Cahn--Hilliard Equation." SIAM Journal on Control and Optimization 57.5 (2019): 3571-3602.

\bibitem{Sid} Siddheshwar, P. G., and R. K. Vanishree. "Lorenz and Ginzburg-Landau equations for thermal convection in a high-porosity medium with heat source." Ain Shams Engineering Journal (2016).

\bibitem{Tuc} H.C. Tuckwell. Random perturbations of the reduced Fitzhugh-Nagumo equation. Physica Scripta, Vol.46, No.6 (1992): 481.

\end{thebibliography}
\end{document}